\documentclass[11pt]{elsarticle}
\usepackage{amsfonts,amsmath,amsthm}

\newcommand{\Z}{\mathbb Z}
\newcommand{\Lee}{\operatorname{Lee}}

\newtheorem{theorem}{Theorem}
\newtheorem{corollary}[theorem]{Corollary}
\newtheorem{proposition}[theorem]{Proposition}
\newtheorem{lemma}[theorem]{Lemma}
\newtheorem{algorithm}{Algorithm}
\newtheorem{example}{Example}

\theoremstyle{definition}
\newtheorem{definition}[theorem]{Definition}

\theoremstyle{remark}
\newtheorem{remark}[theorem]{Remark}

\begin{document}

\begin{frontmatter}

\title{Algebraic Decoding of Negacyclic Codes over ${\mathbb Z}_4$}

\author[ucd]{Eimear Byrne\fnref{fn1}}\ead{ebyrne@ucd.ie}
\author[ucd]{Marcus Greferath\fnref{fn1}}\ead{marcus.greferath@ucd.ie}
\author[cat]{Jaume Pernas\fnref{fn2}}\ead{jaume.pernas@uab.cat}
\author[ucd]{Jens Zumbr\"agel\fnref{fn1}}\ead{jens.zumbragel@ucd.ie}

\fntext[fn1]{The work of E. Byrne, M. Greferath, and J. Zumbr\"agel
  was supported by the Science Foundation Ireland under Grants
  06/MI/006, 08/RFP/MTH1181, and 08/IN.1/I1950.}  
\fntext[fn2]{The work of J. Pernas was partially supported by the
  Spanish MICINN under Grants PCI2006-A7-0616 and TIN2010-17358, and
  by the Catalan AGAUR under Grant 2009SGR1224.}

\address[ucd]{%
  Claude Shannon Institute, School of Mathematical Sciences\\
  University College Dublin, Ireland}

\address[cat]{%
  Departament d'Enginyeria de la Informaci\'o i de les Comunicacions\\
  Universitat Autonoma de Barcelona, Espa\~na}

\begin{abstract}
  In this article we investigate Berlekamp's negacyclic codes and
  discover that these codes, when considered over the integers modulo
  $4$, do not suffer any of the restrictions on the minimum distance
  observed in Berlekamp's original papers \cite{B67,B68}: our codes
  have minimim Lee distance at least $2t\!+\!1$, where the generator
  polynomial of the code has roots $\alpha, \alpha^3, \dots,
  \alpha^{2t-1}$ for a primitive $2n$th root $\alpha$ of unity in a
  Galois extension of ${\mathbb Z}_4$; no restriction on $t$ is
  imposed. We present an algebraic decoding algorithm for this class
  of codes that corrects any error pattern of Lee weight $\leq t$. Our
  treatment uses Gr\"obner bases, the decoding complexity is quadratic
  in~$t$.  
\end{abstract}   


\end{frontmatter}


\section{Introduction}

In his seminal papers \cite{B67,B68}, Berlekamp introduced negacyclic
codes over odd prime fields ${\rm GF}(p)$, and designed a decoding
algorithm that corrects up to $t \le \lfloor \frac{p-1}{2} \rfloor$
Lee errors.  The main idea in Berlekamp's contribution is to represent
error patterns of weight $w$ solely by error locator polynomials of
degree $w$, where the error values are encoded essentially in the
multiplicity of the respective error locations.  Berlekamp's error
locator polynomial satisfies some type of key equation that is solved
during the decoding procedure. Its solution ultimately depends on the
multiplicative invertibility of all odd integers $i \le 2t\!-\!1$ in
(a field extension of) ${\rm GF}(p)$ where $t$ is the maximum Lee
weight of all correctable error patterns. This finally requires $t <
\frac{p+1}{2} $, which is the reason why this idea yields only a very
small class of useful codes.

The project underlying this article revisits Berlekamp's work and
starts with the observation that almost all of the algebra used in the
quoted papers is still valid in a Galois ring, i.e.\ a Galois
extension of the integers modulo $p^m$ where $m$ might be greater
than~$1$. The divisibility condition mentioned above causes problems
if and only if $p$ is odd, and this brought us to the idea to study
codes over $\Z_{2^m}$.

The paper at hand considers the simplest (non-trivial) case, namely
the case where $m=2$, which means we consider negacyclic codes over
$\Z_4$ under the Lee metric. We will show that a negacyclic code is
indeed of minimum Lee distance at least $2t\!+\!1$ if its generator
polynomial has roots $\alpha,\alpha^3,\ldots,\alpha^{2t-1}$ for a
primitive $2n$th root of unity $\alpha$ in a Galois extension of
$\Z_4$. No restriction on $t$ will be imposed.  We present an
algebraic decoding algorithm for this class of codes that corrects any
error pattern of Lee weight $\leq t$.  


\section{Preliminaries}

Throughout this paper, let $R$ denote the Galois ring ${\rm GR}(4,m)$
of characteristic~$4$, order $4^m$, and residue field $K={\rm
  GF}(2^m)$. We let $\mu: {R \longrightarrow K}$, $a \mapsto a + 2R$
be the canonical map from $R$ onto $K$.

The structure of $R$ is well understood (cf.~\cite{R69}).  Its
multiplicative group~$R^{\times}$ has order $2^m(2^m\!-\!1)$ and
contains a unique cyclic subgroup of order $2^m\!-\!1$. This group, in
union with zero, forms the so-called {\em Teich\-m\"uller set} of $R$,
which we denote by ${\cal T}$.  The set ${\cal T}$ forms a complete
set of coset representatives of $2R$ in $R$ and so the image of ${\cal
  T}$ under $\mu$ is the residue field $K$.  Each element $a \in R$
can be expressed in the canonical form $a:=a_0+2a_1$ for suitable
$a_0, a_1 \in {\cal T}$.  The automorphism group of $R$ is cyclic of
order $m$ and with respect to the above canonical form is generated by
the map \[ \pi: R \longrightarrow R \:,\quad a_0 + 2 a_1 \mapsto
a_0^2+2a_1^2 \:. \] Note that for an element $\theta$ of $\cal T$ we
have $\pi( \theta) = \theta^2$.  We remark that $R$ does not contain
an element of order~$4$: suppose $a=a_0+2a_1\in R$ has order~$4$ where
$a_0, a_1\in{\cal T}$, then $a^2 = a_0^2\in{\cal T}$ has order~$2$,
which is impossible as the order of an element in ${\cal T}$ must
divide $2^m\!-\!1$.

\section{Negacyclic Codes Over $\Z_4$}

The following is a BCH-like description of negacyclic codes
over~$\Z_4$, and can be read as the obvious extension of Berlekamp's
work in \cite{B67,B68}.  We outline the theory for the convenience of
the reader, see~\cite{W99} for further details.

\begin{definition}
  Let $n$ be a positive integer.  A {\em negacyclic code} of length
  $n$ over $\Z_4$ is an ideal in the ring $\Z_4[x] /
  \langle x^n \!+\!  1 \rangle$.
\end{definition}


We will work with {\em roots} of a negacyclic code, i.e.\ elements
$\alpha\in R$ satisfying $\alpha^n=-1$.  Note that roots in $R$ exist
only if $n$ is odd: if $n=2\ell$ was even and $\alpha^{2\ell}=-1$,
then $\alpha^{\ell}$ was an element of order~$4$ in $R$, which is
impossible.

Henceforth we will assume that $n$ is odd.  Then there is a primitive
$2n$th root of unity $\alpha$ in $R$ such that $\alpha^n = -1$, i.e.,
$\alpha = -\beta$, where $\beta$ is a primitive $n$th root of unity in
$R$.

Any $\Z_4$-negacyclic code is a principal ideal in $\Z_4[x] / \langle
x^n \!+\!  1 \rangle$, in fact it is generated by a polynomial of the
form $a(b+2) \in \Z_4[x]$ where $x^n+1=abc$ and $a,b,c$ are pairwise
coprime polynomials, in which case the code has size $4^{\delta c
}2^{\delta b}$ where $\delta f$ denotes the degree of the polynomial
$f$ (cf.~\cite[Th.~2.7]{W99}).  There is a natural correspondence
between negacyclic and cyclic codes over $\Z_4$. This is given by the
map
\[ \lambda: \Z_4[x] / \langle x^n \!-\!  1 \rangle 
\longrightarrow \Z_4[x] / \langle x^n \!+\! 1 \rangle \,, 
\quad a(x) \mapsto a(-x) \:. \] 
Clearly, $\lambda$ is ring isomorphism, from which it follows that
any ideal $C$ in $\Z_4[x] / \langle x^n \!-\!  1 \rangle$ is
mapped to an ideal $\lambda(C)$ of $\Z_4[x] / \langle x^n
\!+\!  1 \rangle$.  Moreover, $\lambda$ is an isometry with respect
to the Lee distance, since for every $c = c_0+c_1x+ \cdots +
c_{n-1}x^{n-1} \in \Z_4[x] / \langle x^n \!-\!  1 \rangle$,
we have $\lambda(c) = c_0-c_1x \pm \cdots + c_{n-1}x^{n-1}$ which is
obviously of the same Lee weight as $c$.

\begin{theorem}\label{thmindist}
  Let $C$ be a negacyclic code over $\Z_4$ of odd length $n$
  whose generator polynomial $g$ has the roots $\alpha$, $\alpha^3$,
  \dots, $\alpha^{2t-1}$ for some primitive $2n$th root of unity
  $\alpha \in R$ such that $\alpha^n = -1$.  Then $C$ has minimum Lee
  distance $d_{\Lee}$ at least $2t\!+\!1$.
\end{theorem}

\begin{proof}
  Let $D$ be the pre-image of $C$ under $\lambda$.  Then $D$ is a
  cyclic code of length~$n$, with generator polynomial $f$ satisfying
  $\lambda(f)=g \in \Z_4[x]$. Then $f$ has the roots $\beta,
  \beta^3,\dots,\beta^{2t-1}$ where $\beta = - \alpha$ is a primitive
  $n$th root of unity in $R$.  Now $f \in \Z_4[x]$ is fixed by the
  automorphism $\pi$, so that $0=\pi(f(\theta)) = f(\pi(\theta))$ for
  any root $\theta$ of $f$ in $R$.  Since $\beta$ is contained in the
  Teichm\"uller set of $R$, $f$ also has the roots
  $\pi^j(\beta^i)=\beta^{2^ji}$ for $i \in \{1,3\dots,2t-1\}$.
  Therefore, $f$ has the $2t$ consecutive roots
  $\beta,\beta^2,\dots,\beta^{2t}$.  Therefore a generalization of the
  well-known BCH bound (see for example \cite[Th.~IV.1]{BF02}) shows
  that $D$ has minimum Hamming distance at least $2t\!+\!1$. This
  gives a trivial lower bound on the minimum Lee distance of~$D$. The
  claim now follows from the above isometry observation.
\end{proof}

\begin{remark}
  The lower bound on the Lee distance of negacyclic codes given in
  Theorem~\ref{thmindist} is in general not sharp.  Indeed there are
  codes $C$ with $d_{\Lee}>2t\!+\!1$, as Table~\ref{table} shows.  If
  the actual Lee distance is at least $2r\!+\!1$ with $r>t$ we will
  see in the next section that the key equation carries sufficient
  information to determine all error pattern of Lee weight at
  most~$r$, thus being able to correct up to $r$ errors.  We will then
  present a concrete decoding algorithm for error patterns up to Lee
  weight~$t$.

  \begin{table}\label{table}\centering
    \caption{Parameters of negacyclic codes of length $n$, designed
      error-correcting capability $t$, and rank $k$ (i.e., size $4^k$).}
    \medskip
    \begin{tabular}{ccccc}
      \hline\noalign{\smallskip}
      ~$n$~ & ~~$t$~~ & ~~$k$~~ & $2t\!+\!1$ & ~$d_{\Lee}$~ \\
      \noalign{\smallskip}
      \hline
      \noalign{\smallskip}
      15 & 1 & 11 & 3 & 3 \\
      & 2 & 7 & 5 & 5 \\
      & 3 & 5 & 7 & 10 \\
      \noalign{\smallskip}
      31 & 1 & 26 & 3 & 4 \\
      & 2 & 21 & 5 & 7 \\
      & 3 & 16 & 7 & 12 \\
      & 5 & 11 & 11 & 16 \\
      & 7 & 6 & 15 & 26 \\
      \hline
    \end{tabular}
  \end{table}
\end{remark}

\section{The key equation}

Let $C$ be a negacyclic code with roots $\alpha$, $\alpha^3$, \dots,
$\alpha^{2t-1}$ and minimum Lee distance $d_{\Lee}\ge 2r\!+\!1$.  Let
$v\in\Z_4[z]$ be a received word satisfying $d(v,C)\le r$.  We will
design a decoder to retrieve the unique error polynomial $e$
satisfying $e=v-c$ for some codeword $c$, where $e$ has Lee weight at
most $r$.  Most of what follows will be reminiscent of the according
steps in Berlekamp's papers~\cite{B67,B68}. We will amend the methods
from those sources to the situation at hand.

Let $w$ denote the Lee weight.  We define the error locator polynomial
\begin{equation}\label{eqsigma}
 \sigma = \prod_{i=0}^{n-1} (1-X_iz)^{w(e_i)} \in R[z] \:, 
\end{equation} 
where $X_i=0$ if $e_i=0$, $X_i= \alpha^i$ if $e_i \in \{1,2\}$, and 
$X_i = -\alpha^i = \alpha^{i+n}$ if $e_i=3$.
For each positive integer $k$, we let $s_k$ denote the sum of the
$k$th powers of the reciprocals of the roots of $\sigma$, including
repeated roots, i.e.
\[ s_k = \sum_{j=0}^{n-1} w(e_j) X_j^k \:, \quad k \ge 1 \:. \]
We note that $w(e_j) X_j^k = e_j \alpha^{jk}$ holds for all odd $k$.
Hence, for each $k \in \{1,3,\dots,2t\!-\!1\}$, the $k$th syndrome
$s_k = e(\alpha^k) = v(\alpha^k)$ is known to the decoder.  Let $s$
denote the power series $\sum_{k=1}^{\infty} s_k z^k \in R(z)$. We
have
\[\sigma'(z) = - \sum_{j=0}^{n-1} w(e_j)X_j 
\prod_{i \ne j} (1-X_i z)^{w(e_i)}(1-X_j z)^{w(e_j)-1} \:, \]
and thus  
\begin{align*}
  z \sigma'(z) &=
  - z \sum_{j=0}^{n-1} \prod_{i=0}^{n-1} (1-X_iz)^{w(e_i)} 
  \frac{w(e_j)X_j }{1-X_j z} = - \sigma(z) \sum_{j=0}^{n-1} w(e_j) 
  \sum_{k=1}^{\infty} (X_j z)^k  \\
  &= - \sigma(z)  \sum_{k=1}^{\infty} \Big( \sum_{j=0}^{n-1} w(e_j) X_j^k \Big) z^k 
  = - \sigma(z) \sum_{k=1}^{\infty} s_k z^k = - \sigma(z) s(z)\:.
\end{align*}
Therefore 
\begin{equation}\label{eqnewt}
  s\sigma + z \sigma' =0 \:,
\end{equation}
where the coefficients $s_1,s_3,\dots,s_{2t-1}$ are known to the
decoder.  For any power series $P(z) = \sum_{k=0}^{\infty} P_k z^k \in
R(z)$ we denote the even part and the odd part by $P_e = \sum_{j\ge 0}
P_{2j} z^{2j}$ and $P_o = \sum_{j\ge 0} P_{2j+1} z^{2j+1}$ ,
respectively.  Then the even part and the odd part of equation
(\ref{eqnewt}) read
\begin{gather}
  s_e\sigma_e + s_o\sigma_o + z(\sigma_e)' = 0 \:, \label{eqnewte}\\
  s_e\sigma_o + s_o\sigma_e + z(\sigma_o)' = 0 \:. \label{eqnewto}
\end{gather}
Subtracting $\sigma_e$ times equation~(\ref{eqnewto}) from $\sigma_o$ times
equation~(\ref{eqnewte}) results in the equation
\begin{equation}\label{eqse}
  s_o(\sigma_o^2 - \sigma_e^2) + 
  z \big( (\sigma_e)'\sigma_o - (\sigma_o)'\sigma_e \big) = 0 \:,
\end{equation}
which involves only the odd part of $s$, the latter being known modulo
$z^{2t+1}$. Now let $u = \frac{\sigma_o}{\sigma_e} \in R(z)$ and
rewrite equation~(\ref{eqse}) to obtain
\begin{equation}\label{eqsu}
  s_o(u^2-1)=zu' \:,
\end{equation}
from which we can recursively compute the coefficients
$u_1,u_3,u_5,\dots u_{2t-1}$ via the equations 
\begin{align*}
  u_1 &= -s_1\\
  u_3 &= \frac{-s_3+u_1^2s_1}{3}\\
  u_5 &= \frac{-s_5+u_1^2s_3+2u_1u_3s_1}{5}\\
  \vdots\ &= \quad \vdots
\end{align*}
The reader should notice that this is the point where
Berlekamp's original approach can continue only by imposing a severe
restriction on $t$. In our situation however all the above
denominators are invertible in $R$.

Clearly, $u$ is an odd function and so we may define the power series
$T$ by 
\begin{equation}\label{equT} 
T(z^2) = (1+ zu(z))^{-1} -1. 
\end{equation}
Moreover, the coefficients $T_1,\dots,T_t$ are all known to the
decoder. Next, we define the polynomials $\varphi, \omega \in R[z]$ by
the equations
\begin{equation}\label{eqphiomega}
  \omega(z^2)=\sigma_e(z) \:, \quad \text{ and } \quad
  \varphi(z^2) = \sigma_e(z) + z \sigma_o(z) \:.
\end{equation}
Noting that $1+T(z^2) = \frac{\sigma_e}{\sigma_e+z\sigma_o}$ we
finally arrive at a key equation:
\begin{equation}\label{eqkey}
  (1 + T)\, \varphi \equiv \omega \mod z^{t+1} \:,
\end{equation}
which is the main task of the decoder to solve.

Knowledge of $\varphi$ and $\omega$ is sufficient to recover the error
locations along with their multiplicities.  Using
equation~(\ref{eqphiomega}) we may obtain $\sigma$.  The decoder could
run through the $2n$ roots of unity $1,\alpha,\dots,\alpha^{2n-1}$ and
determine the error polynomial $e$ by
\[ e_j = \begin{cases}
  \,0 & \text{ if } \sigma(\alpha^{-j}) \ne 0 
  \ \text{ and }\ \sigma(\alpha^{-j+n}) \ne 0\\
  \,1 & \text{ if } \sigma(\alpha^{-j}) = 0 
  \ \text{ and }\ \sigma(\alpha^{-j+n}) \ne 0\\
  \,2 & \text{ if } \sigma(\alpha^{-j}) = 0 
  \ \text{ and }\ \sigma(\alpha^{-j+n}) = 0\\
  \,3 & \text{ if } \sigma(\alpha^{-j}) \ne 0 
  \ \text{ and }\ \sigma(\alpha^{-j+n}) = 0\\           
\end{cases} \:. \]
This is easy to see since for each $j \in \{0,\dots,n-1\}$, we have
\begin{gather*}
  \sigma(\alpha^{-j}) = 
  \gamma \prod_{i\ne j} (1 \mp \alpha^{i-j})^{w(e_i)}\\
  \sigma(\alpha^{-j+n}) = \delta \prod_{i\ne j} (1 \pm
  \alpha^{i-j})^{w(e_i)} 
\end{gather*}
where $\gamma = (1\mp 1)^{w(e_j)}\ne 0$ if and only if
$e_j\in\{0,3\}$, $\delta = (1\pm 1)^{w(e_j)}\ne 0$ if and only if
$e_j\in\{0,1\}$, and $1 \pm \alpha^{i-j}$ is a unit whenever $i \ne
j$.

Now we will show that the key equation carries sufficient information
to determine any error pattern of Lee weight at most~$r$.  Let
$B(0,r)$ denote the ball in $\Z_4^n$ centered in $0$ with radius $r$,
and let $\alpha:B(0,r)\to R[z]$ be the function $e\mapsto\sigma$,
mapping an error pattern to its error locator polynomial (see
equation~(\ref{eqsigma})).  Then we consider the function
\[ f: \alpha(B(0,r))\to R^t \:,\quad \sigma\mapsto(T_1,\dots,T_t) \:, \]
where the coefficients $T_1,\dots,T_t$ of the power series $T$ are
obtained as outlined above (see equations~(\ref{eqsu}) and~(\ref{equT})).

\begin{lemma}
  The map $f:\sigma\mapsto(T_1,\dots,T_t)$ is injective on $\alpha(B(0,r))$.
\end{lemma}

\begin{proof}
  Consider the syndrome map $\Z_4^n\to R^t$, $v\mapsto
  (s_1,s_3,\dots,s_{2t-1})$, with $s_k=v(\alpha^k)$.  It kernel equals
  the code $C$ of Lee distance at least $2r\!+\!1$, hence the map is
  injective on $B(0,r)$.  Now we observe that the mappings
  $(s_1,\dots,s_{2t-1})\mapsto (u_1,\dots,u_{2t-1})\mapsto
  (T_1,\dots,T_t)$ of equations~(\ref{eqsu}) and~(\ref{equT}) are
  bijective.
\end{proof}

\begin{proposition}\label{prop:keyeq}
  Let $S:=R[z]/(z^{t+1})$.  For any $T=\sum_{i=1}^tT_iz_i\in S$ there
  is at most one error locator polynomial $\sigma\in\alpha(B(0,r))$
  such that the corresponding key equation $(1+T)\,\varphi=\omega$ in
  $S$ is satisfied, where $\omega(z^2)=\sigma_e(z)$ and
  $\varphi(z^2)=\sigma_e(z)+z\sigma_o(z)$.
\end{proposition}

\begin{proof}
  Suppose that $\sigma\in\alpha(B(0,r))$ satisfies
  $(1+T)\,\varphi=\omega$.  Now $S$ is a local ring with maximal ideal
  $(z)$, and as $\sigma(0)=1$ we have $\varphi(0)=\omega(0)=1$, so
  that $\varphi$ and $\omega$ are units in $S$.  This implies $1+T =
  \omega\varphi^{-1}$, in particular, $T$ is uniquely determined by
  the key equation.  As also $f(\sigma)$ satisfies the key equation by
  construction we have thus $T = f(\sigma)$.  Since $f$ is injective,
  it must hold $\sigma = f^{-1}(T)$, and $\sigma$ is hence uniquely
  determined.
\end{proof}

In the view of Proposition~\ref{prop:keyeq} it remains an open problem
to find the unique solution of the key equation efficiently.  In the
following we assume that $e$ has Lee weight at most $t$, and we
present an efficient decoding method for this case.  

For the classical finite field case, there is a unique pair of coprime
polynomials $[a,b] \in {\rm GF}(p^m)[z]^2$ satisfying the key equation
(\ref{eqkey}) along with the constraints:
\begin{equation}\label{eqdeg}
  a(0) = b(0) = 1 \:, \quad \delta a \le \tfrac{t+1}{2} \:,
  \quad \delta b \le \tfrac{t}{2} \:.
\end{equation}

For the Galois ring case, it is apparent that the required solution
pair $[\varphi,\omega]$ satisfies the constraints~(\ref{eqdeg}).
Although $\varphi$ and $\omega$ are not necessarily coprime in $R[z]$,
we will show in the next section that $2\in R[z]\varphi + R[z]\omega$.
Now over the ring~$R$, a solution $[a,b]$ of the key
equation~(\ref{eqkey}) satisfying $2\in R[z]a + R[z]b$ and the
constraints~(\ref{eqdeg}) will still not be unique in general, but the
modulo 2 solution $[\mu a,\mu b]\in K[z]$ is unique, which will be
sufficient for the decoding problem.

\section{The Ideal Generated by $\varphi$ and $\omega$}

We will show that $2$ can be expressed as a $R[z]$-linear combination
of~$\varphi$ and~$\omega$.  First we note some useful observations.

Let $S$ be a commutative ring with identity $1$.  For $f,g\in S$ we use
the notation $(f,g) := Sf+Sg$ to denote the ideal generated by $f$ and
$g$ in $S$.

\begin{lemma}\label{lem:cp1}
  Let $f,g,h\in S$. Then
  \begin{enumerate}[(a)]
  \item\ $(f,g) = (f,hf\!+\!  g)$,
  \item\ $(h,g) = S$ \ implies \ $(f,g) = (hf,g)$.
  \end{enumerate}
\end{lemma}

\begin{proof}
  We will only prove the inclusion $(f,g)\subseteq (hf,g)$ in (b).
  Since $(h,g)=S$ there are $a,b\in S$ such that $ah+bg=1$, and
  consequently $ahf+bgf=f$.  Now, for all $r,s\in S$ we have
  \[ rf+sg = r(ahf+bgf)+sg = (ra)hf + (rbf+s)g \:. \qedhere \]
\end{proof}

\begin{lemma}\label{lem:cp2}
  Let $a,b,u,v\in S$ and let $f=a+b$, $g=u+v$.
  Suppose that 
  \[ 2b=0,\; \ (f,g)=S,\text{ and } (g,u)=S \:. \] 
  Then $ (fg, au\!+\!  bv) = (f, a)$.
\end{lemma}

\begin{proof}
  First we observe $au+bv = au-bv = ag-fv$.  Next, using
  Lemma~\ref{lem:cp1}, we obtain $(g,ag\!-\!  fv) = (g,fv) = (g,v) =
  (g,u) = S$.  Hence, again using Lemma~\ref{lem:cp1},
  \[ (fg, au\!+\!  bv) = (fg, ag\!-\!  fv) = (f, ag\!-\! 
  fv) = (f, ag) = (f, a) \:. \qedhere \]
\end{proof}

We now specialize to the case that $S=R[z]$ where $R$ is a Galois ring
with residual field $K$.  The following is well-known.

\begin{lemma}\label{lem:cp3}
  Let $f,g$ be polynomials in $R[z]$, then $(f,g)=R[z]$ if and only if
  $(\mu f, \mu g)=K[z]$.
\end{lemma}


Consider the polynomial
\[ \Sigma(z) := \prod_{i=1}^r (1-Y_iz)^{a_i} \in R[z] \:, \] for some
$a_i\in\{1,2\}$ and $Y_i\in R$ such that the $\mu Y_i\in K^{\times}$
are pairwise distinct.  We further let
\[ \tau = \prod_{i=1}^s(1-Y_iz)^2\quad\text{ and }\quad \varepsilon =
\prod_{i=s+1}^r (1-Y_iz) \] be the square and non-square part of
$\Sigma$ (under a suitable re-ordering of the $Y_i$ if necessary).  As
before, we denote the even and the odd part of a polynomial $f\in
R[z]$ by $f_e$ and $f_o$, respectively.

\begin{lemma}\label{lem:pre}
  Given the above notation, there holds $2\tau_o=0$,
  $(\tau,\varepsilon)=R[z]$, and $(\varepsilon,\varepsilon_e)=R[z]$.
\end{lemma}

\begin{proof}
  Since $\tau$ is a square, we have $\mu\tau = \mu\tau_e$. Thus
  $\mu\tau_o=0$ and hence $2\tau_o=0$.  Since $\mu\tau$ and
  $\mu\varepsilon$ have no common factors, we have
  $(\mu\tau,\mu\varepsilon) = K[z]$, and so, by Lemma~\ref{lem:cp3},
  we have $(\tau,\varepsilon) = R[z]$ .

  To show $(\varepsilon,\varepsilon_e)=R[z]$ we simply show that
  $\mu\varepsilon_e$ and $\mu\varepsilon_o$ are coprime.  First we
  note $\mu\varepsilon_e(0)=\mu\varepsilon(0)=1$, and hence $z$ is not
  a common factor of $\mu\varepsilon_o$ and $\mu\varepsilon_e$.
  Suppose now that a (proper) common factor of $\mu\varepsilon_e$ and
  $\mu\varepsilon_o$ exists.  Since both $\mu\varepsilon_e$ and
  $\mu\varepsilon_o/z$ are squares the fact that they have a common
  factor means they have a common factor that is also a square,
  contradicting the fact that $\mu\varepsilon$ is square-free.  Thus
  $\mu\varepsilon_e$ and $\mu\varepsilon_o$ are coprime, and hence, by
  Lemma~\ref{lem:cp3}, $(\varepsilon,\varepsilon_e) =
  (\varepsilon_e,\varepsilon_o) = R[z]$.
\end{proof}

\begin{corollary}\label{cor:tau}\
  $(\Sigma,\Sigma_e) = (\tau,\tau_e)$.
\end{corollary}

\begin{proof}
  We observe that $\Sigma_e =
  \tau_e\varepsilon_e+\tau_o\varepsilon_o$.  Combining
  Lemma~\ref{lem:cp2} and Lemma~\ref{lem:pre} we obtain $(\Sigma,
  \Sigma_e) = (\tau\varepsilon, \tau_e\varepsilon_e\!+\!
  \tau_o\varepsilon_o) = (\tau,\tau_e)$.
\end{proof}

\begin{lemma}\label{lem:square}
  Let $f,g\in R[z]$ be squares.  Then $(fg)_e=f_eg_e$.
\end{lemma}

\begin{proof}
  We have $(fg)_e = f_eg_e+f_og_o$.  Since $f$ and $g$ are squares, as
  in the proof of Lemma~\ref{lem:pre}, it follows that $2f_o = 2g_o =
  0$, and hence $f_og_o=0$.
\end{proof}

\begin{corollary}\label{cor:taue}\
  $\tau_e=\prod_{i=1}^s(1+Y_i^2z^2)$.
\end{corollary}

With these preparations we can prove:

\begin{proposition}\
  $2\in (\Sigma_e,\Sigma_o)$.
\end{proposition}

\begin{proof}
  Observe first that $(\Sigma_e,\Sigma_o) = (\Sigma,\Sigma_e)$ and
  $(\tau,\tau_e)=(\tau,\tau_o)$. Then by Corollary~\ref{cor:tau} it
  suffices to show that $2\in (\tau,\tau_o)$.  Since $2\tau_o=0$ we
  may write $\tau_o=2\rho$ for some regular polynomial $\rho \in
  R[z]$.

  We show that $(\mu\tau, \mu\rho)=K[z]$.  Clearly, the polynomial
  $\mu\tau$ fully splits into linear factors over $K$; its roots are
  $\mu Y_j^{-1}$, $j=1\dots s$.  On the other hand we show that for
  all $j$ we have $\mu\rho(\mu Y_j^{-1})\ne 0$. Using
  Corollary~\ref{cor:taue} we find that
  \[\tau_e(Y_j^{-1}) = 2\prod_{i=1\,,\,i\ne j}^s 
  \left(1+(Y_i\, Y_j^{-1})^2\right) \ne 0\:,\] since
  $\mu(1+Y_i\, Y_j^{-1})\ne 0$ for $i\ne j$.  Hence,
  $\tau_o(Y_j^{-1}) = \tau(Y_j^{-1}) - \tau_e(Y_j^{-1}) =
  -\tau_e(Y_j^{-1}) \ne 0$, and this implies $\mu\rho(\mu Y_j^{-1})\ne
  0$.

  Now, since $(\mu\tau, \mu\rho)=K[z]$ there are $\overline{a},
  \overline{b}\in K[z]$ such that $\overline{a}\,\mu\tau +
  \overline{b}\,\mu\rho = 1$.  Choose $a,b\in R[z]$ with $\mu
  a=\overline{a}$, $\mu b=\overline{b}$.  Then $a\tau + b\rho =
  1+\theta$ for some $\theta\in 2R[z]$, and thus $2a\tau + 2b\rho =
  2$. This proves that $2\in (\tau,2\rho) = (\tau,\tau_o)$.
\end{proof}

\begin{corollary}\label{cor:coprime}\
  $2 \in (\varphi,\omega)$.
\end{corollary}

\begin{proof}
  It is clear that $\sigma$ has the same form as $\Sigma$, defined
  before, and thus $2 \in (\sigma_o,\sigma_e)$.  Moreover
  \[ (\varphi(z^2),\omega(z^2)) = (\sigma_e(z)+z\sigma_o(z),\sigma_e(z))=
  (z\sigma_o(z),\sigma_e(z)) = (\sigma_o,\sigma_e) \:, \] since
  $(z,\sigma_e) = R[z]$.  As $2\in(\sigma_o,\sigma_e)$ there exist
  $a,b\in R[z]$ such that $a\,\varphi(z^2) + b\,\omega(z^2) = 2$. It
  follows $a_e\,\varphi(z^2) + b_e\,\omega(z^2) = 2$.  Therefore we have
  $u\varphi + v\omega = 2$ with $u,v\in R[z]$ such that $u(z^2)=a_e$ and
  $v(z^2)=b_e$.
\end{proof}

\begin{remark}\label{rem:coprime}\
  Suppose that no `double-errors' occurred, i.e., there is no position
  $j$ with $e_j=2$.  Then we have $\tau = 1$, and by
  Corollary~\ref{cor:tau}, we have $(\sigma, \sigma_e) = (\tau,
  \tau_e) = R[z]$.  From this it follows $(\varphi, \omega) = R[z]$, as
  before.
\end{remark}

\section{The Solution Module of the Key Equation}

In this section we investigate the module of solutions to the key
equation~(\ref{eqkey}), $M=\{[a,b]\in R[z]^2\mid a(1+T)\equiv b\mod
z^{t+1}\}$.  First we recall some basic facts on Gr\"obner basis in
$R[z]^2$, further details can be found in \cite{AL94,BF02,BM09}.

\begin{definition}
  Let $\ell$ be an integer. We define a term order $<_{\ell}$ on
  $R[z]^2$ by
  \begin{enumerate}[(a)]
  \item\ $[z^i,0] <_{\ell} [z^j,0]$ and $[0,z^i] <_{\ell} [0,z^j]$ for
    $i<j$,
  \item\ $[0,z^j] <_\ell [z^i,0]$ if and only if $j \le i + \ell.$
  \end{enumerate}
\end{definition}

Let $<$ denote an arbitrary fixed term order. Let $[a,b] \in
R[z]^2\setminus\{0\}$. Then $[a,b]$ has a unique expression as a sum
of monomials $[a,b] = \sum_{i \in I} c_i[z^i, 0] + \sum_{j \in J}
d_j[0, x^j]$ for some finite index sets $I, J$ of nonnegative
integers, and elements $c_i,d_j \in R\setminus\{0\}$. The {\em leading
  term}, $\mathrm{lt}[a, b]$, of $[a, b]$ is then identified as the
greatest term occurring in the above sum with respect to $<$.  The
{\em leading coefficient}, denoted ${\mathrm lc}[a,b]$, is the
coefficient attached to $\mathrm{lt}[a,b]$ and the {\em leading
  monomial} is ${\mathrm lm}[a,b] = {\mathrm lc}[a,
b]\mathrm{lt}[a,b]$.  For any $[a, b],[c,d] \in R[z]^2$ we say that
$[a, b] \preceq [c,d]$ if and only if $\mathrm{lt}[a,b] \le
\mathrm{lt}[c, d]$. Given a set of non-zero elements of $R[z]^2$ there
exists in the set a (not necessarily unique) minimal element with
respect to the quasi-order $\preceq$ associated with $<$. We will
refer to this element as being minimal with respect to $<$.

We say that $\mathrm{lt}[a,b]$ is on the {\em left} (resp.\ {\em
  right}) if $\mathrm{lt}[a,b] = [z^i,0]$ (resp.\
$\mathrm{lt}[a,b]=[0,z^i]$) for some non-negative integer $i$.  A
subset~$\mathcal{B}$ of a submodule~$A$ of $R[z]^2$ is called {\em
  Gr\"obner basis}, if for all $\alpha\in A$ there exists $\beta\in
\mathcal{B}$ such that ${\mathrm lm}(\beta)$ divides ${\mathrm
  lm}(\alpha)$.  The structure of a Gr\"obner basis in $R[z]^2$ is
given by the following lemma (cf.~\cite[Th.\,V.3]{BF02}).

\begin{lemma}\label{lem:gb}
  Let $A$ be a submodule of $R[z]^2$. Suppose that $A$ has elements
  with leading terms on the left and elements with leading terms on
  the right. Then $A$ has a (not necessarily minimal) Gr\"obner basis
  of the form
  \[ \big\{ [a,b], [c,d], [g,h], [u,v] \big\} \] with ${\mathrm
    lm}[a,b]=[z^{i},0]$, ${\mathrm lm}[c,d]=[2z^j,0]$, ${\mathrm
    lm}[g,h]=[0,z^r]$, ${\mathrm lm}[u,v]=[0,2z^s]$ satisfying $i \ge
  j $ and $r \ge s$.  Moreover, the integers $i,j,r,s$ are uniquely
  determined.
\end{lemma}

In \cite[Sec.\,VI]{BF02} an efficient algorithm to compute a Gr\"obner
basis for a submodule $M$ of the form $M=\{ [a,b] \in R[z]^2 \mid a\,U
\equiv b \mod z^r\}$, for some $U \in R[z]$ is given, the so-called
method of Solution by Approximations. This algorithm generalizes one
for the finite field case, derived in \cite{F95}, which can be viewed
as the Gr\"obner basis equivalent of the Berlekamp-Massey algorithm
\cite[Alg.\,7.4]{B68}.  The Solution by Approximations method works
by computing iteratively a Gr\"obner basis of each successive solution
module $M^{(k)}=\{[a,b] \in R[z]^2 \mid a\,U \equiv b \mod z^k\}$,
finally arriving at a basis of $M=M^{(r)}$. The algorithm requires no
searching at any stage of its implementation and has complexity
quadratic in $r$.

We describe this method below, which is particularly simple for the
case of the Galois ring $R$ of characteristic $4$.  As we said before,
the algorithm works by computing iteratively a Gr\"obner basis of each
successive solution module $M^{(k)}=\{[a,b] \in R[z]^2 \mid a\,U
\equiv b \mod z^k\}$, finally arriving at a basis of
$M=M^{(r)}$. Then, the algorithm is basically a method to give the
basis $\mathcal{B}_{k+1} = \{[f'_1,g'_1],\dots,[f'_4,g'_4]\}$ knowing
the basis $\mathcal{B}_k= \{[f_1,g_1],\dots,[f_4,g_4]\}$.

For $\alpha, \beta\in R$ we say that $\alpha$ is a {\em multiple} of
$\beta$ if there exists $x\in R$ such that $\alpha = x \beta$.  This
holds precisely when $\beta\in R^{\times}$ or $\alpha, \beta\in 2R$,
$\beta\ne 0$.

\begin{algorithm}[The Method of Solution by Approximations]\label{alg:sba}
  ~\\
  \indent Input: $U\in R[z]$, $r\in\mathbb{N}$

  Output: A Gr\"obner basis as in Lemma~\ref{lem:gb} of the solution
  module $M=\{[a,b]\in R[z]^2\mid a\,U\equiv b\mod z^r\}$.

  \begin{enumerate}
  \item Let $\mathcal{B}_0 := \{[1,0],[2,0],[0,1],[0,2]\}$ be the
    initial basis of $M^{(0)}$.
	
  \item For each $[f_i,g_i] \in \mathcal{B}_k$, compute the $k^{th}$
    discrepancies $\zeta_i = [f_iU-g_i]_k$, where $[\,\cdot\,]_k$
    denotes the $k$th polynomial coefficient.

  \item For each $[f_i,g_i] \in \mathcal{B}_k$, obtain an element
    $[f'_i,g'_i] \in \mathcal{B}_{k+1}$ as follows.
    \begin{enumerate}
    \item If $\zeta_i = 0$ then $[f'_i,g'_i] := [f_i,g_i]$.
    \item If $\zeta_i \ne 0$ and there is some $[f_j,g_j] \in
      \mathcal{B}_k$ with $\mathrm{lt}[f_j,g_j] <_{\ell}
      \mathrm{lt}[f_i,g_i]$ such that $\zeta_i$ is a multiple of
      $\zeta_j$ then
      \[ [f'_i,g'_i] := [f_i,g_i] - x [f_j,g_j] \:, \] where $\zeta_i
      = x\zeta_j$ (if $\zeta_j\in R^*$ we take
      $x=\zeta_i\zeta_j^{-1}$, and if $\zeta_j=2\varepsilon_j\ne 0$,
      $\zeta_i=2\varepsilon_i$ we take $x=\varepsilon_i\varepsilon_j^{-1}$).
    \item Otherwise, let $[f'_i,g'_i] := [zf_i,zg_i]$.
    \end{enumerate} 

  \item Repeat Steps 2 and 3 for $k=0,\dots,r-1$.

  \item Output $\mathcal{B}_r$.
  \end{enumerate}
  
\end{algorithm}

\begin{example}  
  Let $R = {\rm GR}(4,2) = {\mathbb Z}_4[x]/\langle x^2\!+\!  x\!+\!
  1\rangle$ and let $\alpha=[x]\in R$.  We use Algorithm~\ref{alg:sba}
  to find a Gr\"obner basis of \[ M = \{ [a,b] \in R[z]^2 \mid a\,
  ((3\alpha\!+\!  3)z\!+\!  1) \equiv b\mod z^2\} \] with respect to
  the term order $<_{\ell} \,=\, <_{-1}$.  Hence, $U=(3\alpha+3)z+1$.

  The initial ordered basis of $M^{(0)}$ is
  \[ \mathcal{B}_0 = \{[1,0],[2,0],[0,1],[0,2]\} \:. \] We compute the
  discrepancy for every element in $\mathcal{B}_0$ and find
  $[1,2,3,1]$.  Now, as $[1,0] <_{\ell} [0,1]$ we get $[0,1] -
  \frac{3}{1}[1,0] = [1,1]$ as a new basis element.  Similarly, as
  $[1,0] <_{\ell} [0,2]$ we get $[0,2] - \frac{2}{1}[1,0] = [2,2]$.
  From $[1,0]$ and $[2,0]$ we further get $[z,0]$ and $[2z,0]$.  So,
  the new basis is, after reordering,
  \[ \mathcal{B}_1 = \{[1,1],[2,2],[z,0],[2z,0]\} \:. \] Now, the new
  discrepancies are $[3\alpha \!+\! 3,2\alpha \!+\! 2,1,2]$.  As
  $[1,1]<_{\ell}[z,0]$ and $[1,1]<_{\ell}[2z,0]$ we get new basis
  elements $[z,0]-\frac{1}{3\alpha+3}[1,1] = [z+3\alpha, 3\alpha]$ and
  $[2z,0]-\frac{2}{3\alpha+3}[1,1] = [2z+2\alpha, 2\alpha]$, and from
  $[1,1]$ and $[2,2]$ we get $[z,z]$ and $[2z,2z]$.  Thus finally, the
  founded basis is \[ \mathcal{B}_2 = \{[z\!+\!  3\alpha,3\alpha], [2z\!+\! 
  2\alpha,2\alpha], [z,z], [2z,2z]\} \:. \]
\end{example}

In the next result we establish the minimality of $[\varphi,\omega]$
among the regular elements of the solution module of the key equation
(\ref{eqkey}) with respect to the term order $<_{-1}$.  

\begin{theorem}\label{thm:gb}
  Let $M=\{[a,b] \in R[z]^2\mid a(1\!+\!  T)\equiv b \mod z^{t+1} \}$.
  Let $[a,b] \in M$ such that $\delta a\le \frac{t+1}{2}$, $\delta
  b\le \frac{t}{2}$, and $2 \in (a,b)$.  Suppose further that
  ${\mathrm lc}(a)\in R^{\times}$ if $\delta a > \delta b$ and
  ${\mathrm lc}(b)\in R^{\times}$ if $\delta a \le \delta b$.
  
  \begin{enumerate}[(a)]
  \item Then $[a,b]$ is minimal in $M \setminus M \cap 2R[z]^2$ with
    respect to the term order $<_{-1}$. Moreover, if $[a',b']$ is
    minimal in $M \setminus M \cap 2R[z]^2$ then $[\mu a,\mu b] = \nu
    [\mu a', \mu b']$ for some $\nu\in K^{\times}$.

  \item If in addition $(a,b) = R[z]$ holds, then $[a,b]$ is minimal
    in $M\setminus\{0\}$ with respect to the term order $<_{-1}$, and
    if $[a',b']$ is minimal in $M\setminus M\cap 2R[z]^2$ then $[a,b]
    = \theta [a',b']$ for some $\theta\in R^{\times}$.
  \end{enumerate}
\end{theorem}

\begin{proof}
  Let $[u,v]\in M\setminus\{0\}$ satisfy $\mathrm{lt}[u,v]
  <_{\ell}\mathrm{lt}[a,b]$ for $\ell=-1$.  We will prove $[u,v]\in
  M\cap 2R[z]^2$.  We have $ub=av \mod z^{t+1}$ and first we will
  establish equality in $R[z]$.

  Case 1: $\mathrm{lt}[a,b] = [z^{\delta a},0]$.
  
  If $\mathrm{lt}[u,v] = [z^{\delta u}, 0]$ then $\delta u < \delta a$
  and $\delta v \le \delta u + \ell$, hence $\delta v < \delta a +
  \ell$.

  If $\mathrm{lt}[u,v] = [0, z^{\delta v}]$ then $\delta u + \ell <
  \delta v$ and $\delta v \le \delta a + \ell$, hence $\delta u <
  \delta a$.

  We obtain
  \[ \delta u + \delta b < \delta a + \delta b \le t \quad\text{ and }\quad
  \delta a + \delta v \le 2\delta a + \ell \le t+1+\ell \:. \]

  Case 2: $\mathrm{lt}[a,b] = [0, z^{\delta b}]$.

  If $\mathrm{lt}[u,v] = [z^{\delta u}, 0]$ then $\delta u + \ell < \delta b$
  and $\delta v \le \delta u + \ell$, hence $\delta v < \delta b$.

  If $\mathrm{lt}[u,v] = [0, z^{\delta v}]$ then $\delta u + \ell < \delta v$
  and $\delta v < \delta b$, hence $\delta u + \ell < \delta b$.

  We obtain
  \[ \delta u + \delta b < 2\delta b - \ell \le t - \ell 
  \quad\text{ and }\quad \delta a + \delta v < \delta a + \delta b \le t \:. \]

  For $\ell=-1$ we get $\delta(ub) \le t$ and $\delta(av) \le t$ in
  all cases and therefore $ub=av$ in $R[z]$.
   
  Since $2\in (a,b)$, there exist $f,g\in R[z]$ such that $af+bg = 2$.
  Then $a(fu+gv) = 2u$ and $b(fu+gv) = 2v$.  Suppose that $fu+gv\ne
  0$.  Then, in Case~1 we have $\delta a > \delta b$, thus ${\mathrm
    lc}(a)\in R^{\times}$ by assumption, and we get $\delta a \le
  \delta(a(fu+gv)) = \delta(2u) \le \delta u$, contradicting $\delta u
  < \delta a$.  Similarly, in Case~2 we have $\delta a \le \delta b$,
  thus ${\mathrm lc}(b)\in R^{\times}$, and we get $\delta b \le
  \delta(b(fu+gv) = \delta(2v) \le \delta v$, contradicting $\delta v
  < \delta b$.  Therefore, we have $fu+gv = 0$ and hence $2u=2v=0$.
  It follows $[u,v]\in M\cap 2R[z]^2$, as desired.\medskip

  (a) The above shows that $[a,b]$ in minimal in $M\setminus M\cap
  2R[z]^2$.  Now suppose there exists $[a',b'] \in M\setminus M\cap
  2R[z]^2$ such that $\mathrm{lt}[a',b'] = \mathrm{lt}[a,b]$. 
  We note that
  \[ \mathrm{lc}[a,b] = \begin{cases}
    \mathrm{lc}(a)\in R^{\times} &\text{if }\delta a>\delta b\:,\\
    \mathrm{lc}(b)\in R^{\times} &\text{if }\delta a\le\delta b\:,
  \end{cases} \] and thus $\mathrm{lc}[a,b]$ is a unit.  Hence there
  exist $\theta\in R$ such that $[a',b'] = \theta [a,b] + [r,s]$ and
  $\mathrm{lt}[r,s] < \mathrm{lt}[a,b]$.  From the minimality of
  $[a,b]$ we deduce that $[r,s]\in M\cap 2R[z]^2$, and so $[\mu a',\mu
  b'] = \mu\theta [\mu a, \mu b]$.  Since $[\mu a',\mu b']\ne 0$ we
  have $\mu\theta\ne 0$, and hence $\nu=\mu\theta\in K^{\times}$ and
  $\theta\in R^{\times}$.\medskip

  (b) Since $(a,b) = R[z]$ there exist $f,g\in R[z]$ such that $af+bg
  = 1$.  Let $[u,v]\in M\setminus\{0\}$ satisfy $\mathrm{lt}[u,v]
  <_{-1}\mathrm{lt}[a,b]$ and consider the above proof.  We get
  $a(fu+gv) = u$ and $b(fu+gv) = v$, and as before we can prove $fu+gv
  = 0$.  Hence $u=v=0$, and hence $[a,b]$ is minimal in $M \setminus
  \{0\}$.  Now suppose there exists $[a',b'] \in M\setminus M\cap
  2R[z]^2$ such that $\mathrm{lt}[a',b'] = \mathrm{lt}[a,b]$.  As
  shown in (a) there exist $\theta\in R^{\times}$ such that $[a',b'] =
  \theta [a,b] + [r,s]$ and $\mathrm{lt}[r,s] < \mathrm{lt}[a,b]$.
  Now from the minimality of $[a,b]$ we deduce $[r,s]=0$.
\end{proof}

\begin{corollary}\label{cor:gb}
  Let $M = \{[a,b] \in R[z]^2\mid a(1\!+\!  T)\equiv b \mod z^{t+1}
  \}$, and let $[a',b']$ be the minimal regular element of a Gr\"obner
  basis of $M$.
  \begin{enumerate}[(a)]
  \item Then $[\mu\varphi, \mu\omega] = \nu [\mu a', \mu b']$
    for some $\nu\in K^{\times}$.
  \item If $e$ contains no `double-errors', then $[\varphi, \omega] =
    \theta [a', b']$ for some $\theta\in R^{\times}$.
  \end{enumerate}
\end{corollary}

\begin{proof}
  Let $w:=w(e)=\delta\sigma\le t$ be the number of errors occurred.
  If $w$ is odd, then $\delta\varphi = \frac{w+1}{2}$ and
  $\delta\omega \le \frac{w-1}{2}$, hence $\delta\varphi >
  \delta\omega$; and ${\mathrm lc}(\varphi)\in R^{\times}$.  If $w$ is
  even, then $\delta\omega = \frac{w}{2}$ and $\delta\varphi \le
  \frac{w}{2}$, hence $\delta\varphi \le \delta\omega$; and ${\mathrm
    lc}(\omega)\in R^{\times}$.  By Corollary~\ref{cor:coprime}, we
  have $2\in (\varphi,\omega)$.  So we can apply Theorem~\ref{thm:gb}
  with Remark~\ref{rem:coprime}. 
\end{proof}

We note that since $\omega(0) = \varphi(0) = 1$ we may choose
$[a',b']$ such that $a'(0) = b'(0) = 1$, and then we have
$[\mu\varphi, \mu\omega] = [\mu a', \mu b']$ and $[\varphi, \omega] =
[ a', b' ]$, respectively.

\section{Decoding $\Z_4$-linear Negacyclic Codes}

Let the $\Z_4$-linear negacyclic code $C$ be given as in the previous
sections, and let $v, c, e \in \Z_4[z]$, $\sigma, \sigma_o, \sigma_e,
\varphi, \omega \in R[z]$ and $T \in R(z)$ be given as before.  In
particular, $v=c+e$ with $c\in C$ and the error vector $e$ is of Lee
weight at most $t$.  Let $M = \{[a,b] \in R[z]\mid a\,(1+T) \equiv b
\mod z^{t+1}\}$ be the module of solutions to the key equation
(\ref{eqkey}).  We first compute a Gr\"obner basis of $M$ relative to
the term order $<_{-1}$, which contains an element $[a,b]$ such that
$\mu a = \mu \varphi$ and $\mu b = \mu \omega$.  Then $\mu \varphi,
\mu \omega$ can be used to determine $\mu \sigma = \prod_{i=0}^{n-1}
(1-\mu X_iz)^{w(e_i)} \in K[z]$ via the equations
\[ \mu \sigma = \mu \sigma_e + \mu \sigma_o\:,\ \mu \sigma_e(z) =
\mu\omega(z^2) \:, \text{ and }\ \mu\varphi(z^2) = \mu\sigma_e(z) + z
\mu\sigma_o(z) \:. \] 

Knowledge of $\mu\sigma$ is not sufficient to recover the error
pattern $e$, as errors of the form $e_j=\pm 1$ cannot be
distinguished.  However, by examining the roots of $\mu\sigma$ we find
all error positions, and by examining the double roots we get all
locations $j$ where $e_j=2$ (i.e., the `double-errors').

Let $e^2\in\Z_4^n$ be defined by $e_j^2=2$ if $e_j=2$ and $e_j^2=0$
otherwise.  Note that $e^2$ is completely determined by the roots of
$\mu\sigma$.  Now consider the word $v' := v - e^2 = c + e'$ with $e'
:= e - e^2$.  Then $e'$ does not contain double-errors and has Lee
weight at most~$t$.  Then, using Corollary~\ref{cor:gb}, the error
pattern $e'$ can be found by computing the minimal regular element of
a Gr\"obner basis.

We outline the steps of the algorithm below.

\begin{algorithm}[Algebraic Decoding of $\Z_4$ Negacyclic Codes]
  Let $C$ be a negacyclic code over $\Z_4$ of length $n$,
  whose generator polynomial has roots $\alpha$, $\alpha^3$, \dots,
  $\alpha^{2t-1}$ for a primitive $2n$th root of unity $\alpha \in
  R$ such that $\alpha^n = -1$.\medskip

  Input: $v \in \Z_4[z]$ such that $d(v,C)\le t$

  Output: $c\in C$ such that $w(v-c)\le t$
  
  \begin{enumerate}
  \item Compute the syndromes $s_k := v(\alpha^k)$ for $k =
    1,3,\dots,2t\!-\!1$.
  \item Compute the coefficients $u_k$ using equation (\ref{eqsu}) for
    $k = 1,3,\dots,2t\!-\!1$.  Let $u := \sum_{k=1}^{2t-1} u_k z^k$.
  \item Compute $T(z) \mod z^{t+1}$ from $u$ using equation
    (\ref{equT}).
  \item Obtain a solution $[g,h] \in R[z]^2$ of the key equation
    $a(1+T) = b \mod z^{t+1}$ satisfying the hypothesis of
    Theorem~\ref{thm:gb}.  One way to do this is to identify the
    minimal regular element of a Gr\"obner basis of the solution
    module $M$, relative to the term order $<_{-1}$.
  \item Compute $\mu \sigma(z) = \mu \sigma_e(z) + \mu\sigma_o(z) :=
    \mu h(z^2) + z^{-1}(\mu g(z^2)-\mu h(z^2))$.
  \item Evaluate $\mu\sigma(\mu\alpha^{-j})$ for $j=0,\dots,n\!-\!1$.
    \begin{itemize}
    \item If $\mu\alpha^{-j}$ is a double root of $\mu\sigma$ then
      $e_j=2$.
    \item If $\mu\alpha^{-j}$ is a single root of $\mu\sigma$ then $e_j
      \in \{ \pm 1 \}$.
    \end{itemize}
  \item Let $e^2 := \sum_{j, e_j=2} 2z^j$, and let $v' := v - e^2$.
  \item Repeat Steps 1.--4. with $v'$ in place of $v$, and compute
    $\sigma'(z) = \sigma'_e(z) + \sigma'_o(z) := h(z^2) +
    z^{-1}(g(z^2)-h(z^2))$.
  \item Compute $e'$ by evaluating $\sigma(\alpha^{-j})$ 
    and $\sigma(\alpha^{-j+n})$ for $j=0,\dots,n\!-\!1$.
    \begin{itemize}
    \item If $\alpha^{-j}$ is a root of $\sigma$ then $e'_j=1$.
    \item If $\alpha^{-j+n}$ is a root of $\sigma$ then $e'_j=3$.
    \end{itemize}
  \item Output $c := v' - e'$.
  \end{enumerate}
\end{algorithm}

We will conclude our work by a concrete example.
 
\begin{example}
  Let $R = {\rm GR}(4,4) = \Z_4[x] / \langle x^4 \!+\!  2x^2
  \!+\!  3x \!+\!  1 \rangle$ and let $\alpha = [x]\in R$.  We use a
  code of length $n=15$ with $t=2$.  Let the received word be $v = 2 +
  z + 3z^2 + 2z^4 + z^5 + 2z^6 + 3z^7 + z^8 + 3z^9 + z^{10} +
  3z^{13}$.

  \begin{enumerate}
  \item The list of syndromes is $[ 3\alpha^3 + \alpha^2 + 3\alpha +
    2, \,2\alpha^3 + \alpha^2 + 2\alpha + 1 ]$.
    \addtocounter{enumi}{1}
  \item The $T$ polynomial of the key equation is $T(z) = (2\alpha^3 +
    \alpha^2 + \alpha)z^2 + (3\alpha^3 + \alpha^2 + 3\alpha + 2)z +
    1$.
  \item The solution of the key equation is $[(\alpha^3 + 2\alpha^2 +
    3\alpha + 3)z + 3\alpha^3 + 3\alpha^2 + 2\alpha + 3,\, z +
    3\alpha^3 + 3\alpha^2 + 2\alpha + 3)] \in R[z]^2$.
    \addtocounter{enumi}{2}
  \item We find $e^2=0$.  \addtocounter{enumi}{1}
  \item Repeating the process with $v'=v$ we find 
    $e'=z^4-z^{13}$ and $c=v'-e'$.
  \end{enumerate}
\end{example}

\end{document}